\documentclass{amsart}
\usepackage{amsthm,amssymb,amsmath,dsfont,bigdelim}
\usepackage{graphicx}
\usepackage{float}
\usepackage[arc,all]{xy}
\usepackage{longtable}
\newtheorem{thm}{Theorem}[section]

\newtheorem{lem}[thm]{Lemma}
\newtheorem{prop}[thm]{Proposition}
\theoremstyle{definition}
\newtheorem{defn}[thm]{Definition}

\newtheorem*{proof M}{\textbf{Proof of  Main Theorem}}
\theoremstyle{remark}

\numberwithin{equation}{section}

\begin{document}
\title[ superbreadth]{On the nilpotent Lie superalgebras of small superbreadth  }%
\author[A. Shamsaki]{afsaneh shamsaki}%
\address{School of Mathematics and Computer Science,
Damghan University, Damghan, Iran}
\email{Shamsaki.afsaneh@yahoo.com}%
\author[P. Niroomand]{Peyman Niroomand}\thanks{Corresponding Author: Peyman Niroomand}
\address{School of Mathematics and Computer Science,
Damghan University, Damghan, Iran}
\email{niroomand@du.ac.ir, p$\_$niroomand@yahoo.com}
\author[M. Ladra]{Manuel Ladra}%
\address{Department of Mathematics, Institute of Mathematics, Universidade de Santiago de Compostela, 15782 Santiago de Compostela, Spain }
\email{manuel.ladra@usc.es}

\begin{abstract}
In this paper, we classify finite-dimensional nilpotent Lie superalgebras of superbreadth at most two.
\end{abstract}
\maketitle
\section{ Introduction}
The classification of algebraic objects is a classical problem in mathematics. The classification of nilpotent Lie algebras has attracted considerable attention from many
authors. In general, classifying all non-isomorphic nilpotent Lie algebras in finite dimensions is difficult, and only nilpotent Lie algebras of small dimension, at most $ 7 $, have been characterized over the complex numbers. The large number of finite-dimensional Lie algebras arising in these classifications makes the classification problem difficult. To make the classification
attainable, many authors impose additional conditions on nilpotent Lie algebras. One such condition is the breadth of a nilpotent Lie algebra. The breadth of a Lie algebra was introduced by Leedham-Green, Neumann and Wiegold. The breadth $ b(L) $ of a Lie algebra $ L $ is defined as the maximum of the dimensions of the images of $ ad_x $ for all $ x\in L. $
The finite-dimensional nilpotent Lie algebras of breadth $ 1, $ $ 2 $ and $ 3 $ are characterized in \cite{b2,b3}. Lie superalgebras have long been of interest in mathematical physics. In recent work, various authors have sought to extend familiar results from Lie algebras to the setting of Lie superalgebras. The classification of all finite-dimensional nilpotent Lie superalgebras remains an open problem, as is the case for finite-dimensional
nilpotent Lie algebras. \\
In this paper, we classify finite-dimensional nilpotent Lie superalgebras of superbreadth at most two.\\ 
We regard $ \mathbb{Z}_{2}=\lbrace \overline{0}, \overline{1} \rbrace $ as a field.
Let $ V_{\overline{0}} $ and $ V_{\overline{1}} $ be vector spaces. Then a $ \mathbb{Z}_{2} $-graded vector space $ V $ is a direct sum of $ V_{\overline{0}} $ and $ V_{\overline{1}} $ such that $ V=V_{\overline{0}}\oplus V_{\overline{1}}. $ Such a space is also called a superspace. Elements of $ V_{\overline{0}} $ and $ V_{\overline{1}} $ are called even and odd, respectively. Non-zero elements of $ V_{\overline{0}} \cup V_{\overline{1}}$ are called homogeneous elements; for $ v\in V_\sigma $ such that $ \sigma \in \mathbb{Z}_2. $ Then $ | v |=\sigma $ is the degree of $ v. $ A superalgebra is a $ \mathbb{Z}_{2} $-graded algebra $ V=V_{\overline{0}}\oplus V_{\overline{1}} $ (that is, if $ x\in V_{\alpha}, $ $ y\in V_{\beta}, $ $ \alpha, \beta\in V_{\alpha+\beta}, $ then $ xy\in V_{\alpha+\beta} $).  \\ 
A Lie superalgebra $ L=L_{\overline{0}}\oplus L_{\overline{1}} $ is a superalgebra equipped with a bilinear map $ [\cdot, \cdot] : L\times L \rightarrow L $ that satisfies the following identities.
\begin{itemize}
\item[(i).] $[L_{\sigma_1}, L_{\sigma_2}]\subset L_{\sigma_1+\sigma_2}$ for $ \sigma_1, \sigma_2\in \mathbb{Z}_{2} $ ($\mathbb{Z}_{2}$-grading),
\item[(ii).] $ [x, y]=-(-1)^{\mid x\mid \mid y \mid} [y, x]$ (graded skew-symmetry),
\item[(iii).] $ (-1)^{\mid x\mid \mid z \mid} [x,[y, z]]+(-1)^{\mid y\mid \mid x \mid} [y,[z, x]]+ (-1)^{\mid z\mid \mid y \mid} [z,[x, y]]=0 $ (graded Jacobi identity) for all $ x, y, z\in L. $

\end{itemize}
The even part $ L_{\overline{0}} $ is a Lie algebra and $ L_{\overline{1}} $ is a $ L_{\overline{0}} $-module. If $ L_{\overline{1}}=0, $ then $ L $ is a Lie algebra. In general, however, a Lie superalgebra need not be a Lie algebra. A Lie superalgebra with zero even part is an abelian Lie superalgebra; i.e., $ [x,y]=0 $ for all $ x, y\in L. $ A subalgebra of $ L $ is a $ \mathbb{Z}_{2} $-graded vector subspace that is closed under the bracket operation. A $ \mathbb{Z}_{2} $-graded vector subspace $ I $ is a graded ideal of $ L $ if $ [I, L]\subseteq I. $ The center of $ L $
is a graded ideal $ Z(L)=\lbrace z\in L \mid [z,x]=0~~ \text{for all } x\in L\rbrace.$ Clearly, if $ I $ and $ J $ are graded ideals of $ L, $ then so is $ [I, J]. $
\section{Classification of nilpotent Lie superalgebras with $  b(L)\leq 1$  }
In this section, we define the superbreadth $ b(L) $ of a finite-dimensional Lie superalgebra $ L $ and classify finite-dimensional nilpotent Lie superalgebras
$ L $ with $ b(L)\leq 1. $
\begin{defn}
Let $ L $ be a finite-dimensional Lie superalgebra. The superbreadth $ b(L) $ is the maximum of the dimensions of the images of $ ad_{x} $ for all $ x\in L. $
\end{defn}
Let $ A $ be an ideal of $ L $ and $ b_{A}(x)=$rank~$ ad_{x}|_{A}. $ Then
\begin{equation*}
b_{A}(L)=\max \lbrace b_{A}(x) \mid x\in L \rbrace.
\end{equation*}
Throughout this paper, when a Lie superalgebra $L = L_{\bar{0}} \oplus L_{\bar{1}}$ is of dimension $r + s$, in which $\dim L_{\bar{0}} = r$ and $\dim L_{\bar{1}} = s$, we write $\dim L = (r,s)$. We use analogous notation for the superbreadth $b(L).$\\
The following lemma and proposition classify finite-dimensional nilpotent Lie superalgebras $ L $ with $ b(L)\leq 1. $ 
\begin{lem}\label{l0}
Let $ L $ be a finite-dimensional Lie superalgebra. Then $ L $ is an abelian Lie superalgebra  if and only if $ b(L)=(0,0). $
\end{lem}
\begin{proof}
This follows directly from the definition of superbreadth.
\end{proof}
\begin{prop}\label{pr1.1}
Let $ L $ be a finite-dimensional Lie superalgebra. Then $ b(L)=(r,s) $ such that $ r+s=1 $ if and only if $ \dim [L, L]=(n,m) $ such that $ m+n=1. $
\end{prop}
\begin{proof}
Let $ L $ be a finite-dimensional Lie superalgebra with $ b(L)=(r,s) $ such that  $ r+s=1. $ Then $ b(L) $ is either $ (0,1) $ or $ (1,0). $ 
First, let $ b(L)=(0,1). $ Then $ \dim [L,L]=\dim [L_{\overline{0}}, L_{\overline{1}}] \neq 0$ and $ \dim [L_{\overline{0}}, L_{\overline{0}}]=\dim [L_{\overline{1}}, L_{\overline{1}}]=0. $ Suppose, to the contrary, that $ \dim [L_{\overline{0}}, L_{\overline{1}}] \geq 2.$  Since  $ \dim [L_{\overline{0}}, L_{\overline{1}}] \geq 2,$  we may choose $ x_0, y_0\in L_{\overline{0}}, $  $ x_1, y_1\in L_{\overline{1}} $ and $ [x_0, x_1]=z_1, $ $  [y_0, y_1]=w_1 $ such that the set $ \lbrace z_1, w_1\rbrace \subseteq  L_{\overline{1}} $ is linearly independent. Since $ b(L)=(0,1), $ we have 
\begin{equation*}
b(x_1)=b(y_1)=b(x_0)=b(y_0)=(0, 1). 
\end{equation*}
Hence $[x_0, y_1]=0 $ and  $  [y_0, x_1]=0. $ Then $ [x_0+y_0, x_1]=z_1 $ and  $[x_0+y_0,y_1]=w_1 $ imply that $ b(x_0+y_0)=(n', m')$ with $n'+m'\geq 2. $ This is a contradiction. Therefore $ \dim [L, L]=(0,1). $ If $ b(L)=(1,0), $ then $ \dim [L, L]=(1,0) $ by a similar argument.\\  
Conversely, suppose that $ \dim [L, L]=(0, 1). $ Then $ \dim [L_{\overline{0}}, L_{\overline{0}}]=\dim [L_{\overline{1}}, L_{\overline{1}}]=0 $ and $ \dim [L_{\overline{0}}, L_{\overline{1}}]\neq 0. $ We also have 
 $ b(L)\leq \dim [L,L]=\dim[L_{\overline{0}}, L_{\overline{1}}]=1  $ by definition. Since $ L $ is a non-abelian Lie superalgebra,  $ b(L)\neq (0,0) $ by Lemma \ref{l0}. Since   $ \dim [L,L]=\dim [L_{\overline{0}}, L_{\overline{1}}] \neq 0$ and $ \dim [L_{\overline{0}}, L_{\overline{0}}]=\dim [L_{\overline{1}}, L_{\overline{1}}]=0, $  we have  $ b(L)=(0,1). $  If $ \dim [L, L]=(1,0),$  then similarly $ b(L)=(1,0) $. Therefore the result follows.
\end{proof}
The next theorem determines the structure of finite-dimensional nilpotent Lie superalgebras $ L $ with $ b(L)=(r,s) $ such that $ r+s=1. $
\begin{thm}
Let $ L $ be a finite-dimensional nilpotent Lie superalgebra and   $ b(L)=(r, s). $ Then  $ r+s=1 $ if and only if  $ L $ is isomorphic to one of the   following nilpotent Lie superalgebras
\[ H_e=\langle x_1, \dots, x_m, x_{m+1}, \dots, x_{2m}, z \rangle \oplus \langle y_1, \dots, y_n  \rangle \oplus A(k)
\] 
with $ [x_i, x_{i+m}]=z $ and $  [y_j, y_j]=z  $ for all  $ i, $  $ j $ such that $1 \leq i\leq m, 1\leq j\leq n $   
or 
\[H_o=\langle x_1, \dots, x_{m}  \rangle \oplus \langle y_1,\dots, y_m, z  \rangle \oplus A(k)
\] 
with $ [x_i, y_i]=z$ for all  $ i $   such that $1 \leq i\leq m, $ where $ A(k) $ is an abelian Lie superalgebra for $ k\geq 0. $ 
\end{thm}
\begin{proof}
Let $ L $ be a Lie superalgebra with $ b(L)=(r, s) $ such that $ r+s=1. $ Then $ \dim [L,L]=(n_1, m_1) $ such that $  n_1+m_1=1 $ by Proposition \ref{pr1.1}. Hence $ L $ is isomorphic to $ H_e $ or $ H_o $ by \cite[Lemma 2.2]{sh}, \cite[Page 4]{Na2} and \cite[Proposition 1]{Rodr}. The converse is immediate.
\end{proof}
The following theorem will be used in the next section.
\begin{thm}\label{th2.6}
Let $ L $ be a finite-dimensional Lie superalgebra with $ b(L)=(r,s) $ such that $ r+s=n>0. $  Then there exists $ x\in L $ such that $ b(x)=n $ and 
\begin{itemize}
\item[(i).] if $ [x, x]\neq 0, $ then $ \dim L/Z(L)\geq n, $
\item[(ii).] if $ [x, x]= 0, $ then $ \dim L/Z(L)\geq n+1. $
\end{itemize}
\end{thm}
\begin{proof}
Let $ L $ be a finite-dimensional Lie superalgebra  with $ b(L)=(r,s)$ such that $ r+s=n>0. $ Then there exists $ x\in L $ such that $ b(x)=n. $ Since $ b(x)=n, $ there exist $ x_1, \dots, x_n \in L$ such that the set $ \lbrace z_i \mid   [x, x_i]=z_i ~~\text{ for all}~~  1\leq i\leq n \rbrace$ is linearly independent. If $ [x,x]\neq 0, $ then we show that $ \lbrace x_i+Z(L) \mid 1\leq i\leq n \rbrace $ is linearly independent. Let
\begin{equation*}
\alpha_1 (x_1+Z(L))+ \dots + \alpha_n (x_n+Z(L))=Z(L)
\end{equation*}
for all $ \alpha_1, \dots, \alpha_n \in \mathbb{F}. $ Hence $ \alpha_1 x_1+ \dots + \alpha_n x_n \in Z(L)$ and
\begin{align*}
0&=[x, \alpha_1 x_1+\dots + \alpha_{n}x_{n}] \cr
&=\alpha_1[x,x_1]+\dots +\alpha_{n}[x, x_{n}] \cr
&=\alpha_1 z_1+\dots + \alpha_{n} z_{n}.
\end{align*}
Since the set $ \lbrace z_i \mid  1\leq i\leq n \rbrace $ is linearly independent, we have $ \alpha_1=\dots =\alpha_{n}=0. $ Therefore 
$ \lbrace x_i+Z(L) \mid  1\leq i\leq n \rbrace $ is linearly independent and hence $ \dim L/Z(L)\geq n. $ If $ [x,x]=0, $ then $ \lbrace x+Z(L), x_i+Z(L) \mid 1\leq i\leq n \rbrace $ is linearly independent by a similar argument, and hence $ \dim L/Z(L)\geq n+1. $ 
\end{proof}
\section{Classification of Lie superalgebras  $ L $ with $  b(L)=2$ }
In this section, we obtain the classification of finite-dimensional nilpotent Lie superalgebras  $ L $ with $ b(L)=2. $ We start with the following lemmas. 
\begin{lem}\label{lem3.1}
Let $ L $ be a finite-dimensional nilpotent Lie superalgebra with $ b(L)=(0,r) $ such that $ r\geq 2 $ and let $ A $ be a maximal abelian ideal of $ L. $ Then $ C_{L}(A)=A. $
\end{lem}
\begin{proof}
We have $ A\subseteq C_{L}(A). $ Suppose, to the contrary, that there exists $ x\in  C_{L}(A)\setminus A.$ Since $ C_{L}(A) $ is an ideal of $ L $ and $ L $ is nilpotent, $ L $ acts nilpotently on $ C_{L}(A)  /A.$ By Engel's theorem, we have $ [L, x] \in A.$ On the other hand, since $ b(L)=(0,r) $ for $ r\geq 2 $, we have  $ \dim [L,L]=\dim [L_0, L_1] $ and so $ [x, x]=0. $ Hence  $ A'=\langle x, A \rangle $ is an abelian ideal of $ L $ such that $ A\subset A'. $ This yields a contradiction. Therefore $ C_{L}(A)=A. $
\end{proof}
\begin{lem}\label{lem3.2}
Let $ L $ be a finite-dimensional Lie superalgebra with $ b(L)=(r, 0)$  such that $ r\geq 2 $ and let $ A $ be a maximal abelian ideal of $ L. $ If $ [L_{\overline{0}}, L_{\overline{0}}]=0, $ then $ [x, x]\neq 0 $ for all $ x\in L_{\overline{1}}
 \setminus A_{\overline{1}}. $
\end{lem}
\begin{proof}
 Suppose, to the contrary, that there exists $ x\in  L_{\overline{1}}
 \setminus A_{\overline{1}} $ such that $ [x,x]=0. $  Since $ b(L)=(r,0) $ for $ r\geq 2 $, we have $ \dim [L_{\overline{0}}, L_{\overline{1}}]=0$ and so $ [x, L_{\overline{0}}]=0. $   On the other hand, $ Z(L)\subseteq A $ and $ x\notin A,$ so we have  
 $ [x,L_{\overline{1}}] \in L_{\overline{0}}. $ 
   Put $ H=\langle x, L_{\overline{0}}\rangle. $ Thus $ H $ is an abelian ideal; since   $ A $ is a maximal abelian ideal of $ L, $ we have $ H\subseteq  A.$
Since $ x\notin A, $ this is a contradiction.  Hence $ [x,x]\neq 0 $ for all $ x\in L_{\overline{1}} \setminus A_{\overline{1}}.  $ 
\end{proof}
\begin{lem}\label{lem3.3}
Let $ L $ be a finite-dimensional nilpotent Lie superalgebra with $ b(L)=(r, s)$ such that $ r\geq 1, $ $ s\geq 0 $  and let $ A $ be a maximal abelian ideal of $ L $ with $ b_{A}(L)=(0,0). $ Then   $ \dim [L_{\overline{0}}, L_{\overline{0}}]=0. $
\end{lem}
\begin{proof}
Since $ b_{A}(L)=(0,0), $ we have $ C_{L}(A)=L $ and so $ A=Z(L). $ Thus
 $ C_{L_{\overline{0}}}  (A)=L_{\overline{0}}\setminus A_{\overline{0}}  $ and $ C_{L_{\overline{1}}} (A)=L_{\overline{1}}. $ We claim that
 $ [L_{\overline{0}}, L_{\overline{0}}]=0. $ Since $ A $ is abelian,  $ A_{\overline{0}} \subseteq C_{L_{\overline{0}}} (A). $ Let 
  $ C_{L_{\overline{0}}}(A)\neq A_{\overline{0}}. $ Then there exists $ x\in C_{L_{\overline{0}} }(A)=L_{\overline{0}}\setminus A_{\overline{0}}$ and so  $ x\in C_{L}(A)=L\setminus A.$ Now, by Engel's Theorem  $ [L, x] \subseteq A. $ Then
   $ A'= \langle x, A \rangle $ is an abelian ideal of $ L. $ Hence $ A \subseteq A' \subseteq L $ and $ A' $ is a maximal abelian ideal of $ L. $ Also, $ A $ is a maximal abelian ideal of $ L $ and  we have $ L=A'. $
Since $ L $ is non-abelian, this yields a contradiction. Hence $ C_{L_{\overline{0}}} (A)= A_{\overline{0}}=L_{\overline{0}}$
 and so $ [L_{\overline{0}}, L_{\overline{0}}]=0. $
\end{proof}

\begin{prop}\label{Pr3.4}
Let $ L $ be a finite-dimensional nilpotent Lie superalgebra with $ b(L)=(2,0) $   and let $ A $ be a maximal abelian ideal such that $ b_{A}(L)=(0,0). $ Then one of the following conditions holds. 
\begin{itemize}
\item[(i).] $ \dim [L,L]=(2,0), $
\item[(ii).] $ \dim [L,L]=(3,0) $ and $ \dim L/Z(L)=(0,2). $
\end{itemize}
\end{prop}
\begin{proof}
Since $ b_{A}(L)=(0,0) $ and $ A $ is a maximal abelian ideal, we have $ A=Z(L). $
Since $ b(L)=(2,0) $, $ \dim [L_{\overline{0}}, L_{\overline{0}}]=0 $ by Lemma \ref{lem3.3} and $ \dim[L_{\overline{0}}, L_{\overline{1}}]=0. $ 
  Also, $ [x,x]\neq 0 $ for all $ x\in L_{\overline{1}}\setminus Z(L) $ by Lemma \ref{lem3.2}.
   Since $ b(L)=(2,0) $ and $ \dim [L, L]=\dim [L_{\overline{1}}, L_{\overline{1}}], $ there exists $ a\in L_{\overline{1}}\setminus Z(L) $ with $ b(a)=(2,0) $ and  $ [a,a]\neq 0. $ 
  Hence $ [a, a]=z_1 $ and $ [a, x]=z_2 $ for  $ a, x\in L_{\overline{1}}, $  $ z_1, z_2\in L_{\overline{0}} $ and  $\lbrace  z_1, z_2 \rbrace$ is a linearly independent set. Since $ x\in  L_{\overline{1}}\setminus Z(L)$ and  $ [L_{\overline{0}}, L_{\overline{0}}]\neq 0, $ we have $ [x, x]\neq 0 $ by Lemma \ref{lem3.2}. We consider the following cases. \\
 Case 1. If $ b(x)=(1, 0), $ then $ [x, x]=z_2. $ We show that there is no $ y\in L_{\overline{1}}\setminus Z(L) $ such that $ [y, y]=z_3 $ for all $ z_3 \in L_{\overline{0}}$ and $\lbrace z_1, z_2, z_3\rbrace $ is a linearly independent set.  Suppose, to the contrary, that there exists such an element.\\
Subcase 1-1. Let $ b(y)=(1,0). $ Since $ b(x)=(1,0), $  $ b(a)=(2,0), $ and  $\lbrace  z_1, z_2, z_3 \rbrace$ is a linearly independent set,
we have $ [y, a]=[y, x]=0. $
  It follows directly that $ b(a+x+y)=(3,0). $ This is a contradiction.\\
Subcase 1-2. Let $ b(y)=(2,0). $  Since $ b(x)=(1,0) $ and $ b(a)=(2,0), $  we have $ [y, x]=\alpha z_2 $  and $  [y, a]=\beta z_1+\gamma z_2$ for $ \alpha, \beta, \gamma \in \mathbb{F}. $  If  $ [y, a]=0 $  and $  [y, x]=\alpha z_2\neq 0 $ or $ [y, a]=[y, x]=0, $   then $ b(\alpha a+y)=(3,0) $ for $ \alpha \in \mathbb{F} $ and $ \alpha\neq 0. $ In the case $ [y, a] $  and $  [y, x] $ are non-zero since $ b(y)=(2,0), $ $ [y, x]=\alpha z_2 $  and $  [y, a]=\gamma z_2 $ for $ \alpha, \gamma \in \mathbb{F}. $ Hence $ b( a+y)=(3,0). $  If  $ [y, x]=0 $  and $  [y, a]=\beta z_1+\gamma z_2 $ for $ \beta, \gamma \in \mathbb{F}, $ then $ b(x+y)=(3,0). $
  Therefore, this yields a contradiction, and hence $ \dim [L, L]=(2,0). $ \\
 Case 2. Let $ b(x)=(2,0). $ If $ [x,x]=z_1, $ then  $ \dim [L, L]=(2,0) $  as in Case 1.
 Let $ [x,x]=z_2. $ Since $ b(x)=(2,0)$ and $ x\in L_{\overline{1}}\setminus Z(L), $ there exists $ y\in L_{\overline{1}} \setminus Z(L)$ such that $ [x, y]=z_3. $
  Hence $ b(\alpha x+a)=(3,0) $ for $ \alpha \in \mathbb{F}  $ and $ \alpha \neq 0 $; this is a contradiction.  Thus $ \dim [L, L]=(2,0). $\\
 If $ [x,x]=z_3 $ such that $ \lbrace z_1, z_2, z_3 \rbrace $ is a linearly independent set, then we claim that there is no element  $ y \in L_{\overline{1}}$ such that $ [y, y]=z_4 $ and  $ \lbrace z_1, z_2, z_3, z_4 \rbrace $ is a linearly independent set. Suppose, to the contrary, that there exists such an element. We have the following subcases.  \\
 Subcase 2-1. Assume that $ b(y)=(1,0). $ Since $ b(x)=b(a)=(2,0) $ and  $ b(y)=(1,0), $ we have $ [a, y]= [x, y]= 0.$ Thus $ b(a+y)=(3,0).$ This is a contradiction.\\
 Subcase 2-2. Let $ b(y)=(2,0). $ Since $ [a, a]=z_1, $  $ [a, x]=z_2 $ and $ [x, x]=z_3, $ we have
  $ [a,y]=\alpha z_2+ \beta z_1 $ and $ [x,y]=\gamma z_2+\delta z_3 $ for $ \alpha, \beta, \gamma, \delta \in \mathbb{F}. $
  Since $ b(y)=(2,0), $ certain coefficient configurations $ \alpha, \beta, \gamma, \delta \in \mathbb{F} $ do not occur. Considering the possible coefficients $ \alpha, \beta, \gamma, \delta \in \mathbb{F}$, and arguing as in subcase 1-2, we obtain $ b(y)=(3,0) $ or $ b(y+a)=(3,0). $ This is a contradiction.\\
  Therefore $ \dim [L,L]=(3,0) $ in subcases 2-1 and  2-2. We now prove that if  $ \dim [L,L]=(3,0), $ then
   $ \dim L/Z(L)=(0,2). $ We next show that there is no $ y\in L_{\overline{1}}\setminus  \langle x, a \rangle  $ such that $ [y,y] $ equals $ z_1, $ $ z_2 $ or $ z_3. $ On the contrary, let $ [y, y]=z_1 $ and $ b(y)=(1,0). $ Then $ b(y+x)=(3,0). $ This is a contradiction. In the case  $ [y, y] $ equals $ z_2 $ or $ z_3 $ and $ b(y)=(1,0)$  this yields a contradiction by a similar argument.
 Assume that $ [y, y]=z_1 $ and $ b(y)=(2,0). $ Since $ b(a)=b(x)=(2,0), $ $ [x, y]=\alpha z_1 $ and $ [a,y]=\beta_1 z_1+\beta_2 z_2 $ 
 for $ \alpha, \beta_1, \beta_2\in \mathbb{F}. $
Considering the possible coefficient cases, there is an element $ w $ such that $ b(w)\geq 3. $ This is a contradiction. Hence $ x, a \in L_1 $ are the only elements in  $ L_{\overline{1}} \setminus Z(L). $  Since  $ [L_{\overline{0}}, L_{\overline{1}}]=0 $ and  by Lemma \ref{lem3.3},  we have  $ L_{\overline{0}} \subseteq Z(L).$ Hence  $ \dim L/Z(L)\leq 2. $  On the other hand, 
$2\leq \dim L/Z(L) $  by Theorem \ref{th2.6} (i), and since $ \dim (L/Z(L))_{\overline{1}}=2 $, it follows that $ \dim L/Z(L)=(0,2). $
Therefore   $ \dim [L, L]=(3,0) $ and $ \dim L/Z(L)=(0,2). $
\end{proof}
\begin{prop}\label{pr3.3}
There is no  finite-dimensional nilpotent Lie superalgebra  $ L $ with $ b(L)=(0,2) $   and  a maximal abelian ideal $ A $ such that $ b_{A}(L)=(0,0). $ 
\end{prop}
\begin{proof}
Since $ b_{A}(L)=(0,0)$ and $ A $ is a maximal abelian ideal of $ L, $ we have $ A=Z(L) $ and so $ C_{L}(A)=L. $
$ C_{L}(A)=A $ by Lemma \ref{lem3.1}, which is a contradiction. Hence no such Lie superalgebra exists.
\end{proof}
\begin{lem}\label{ll}
Let $ L $ be a finite-dimensional nilpotent Lie superalgebra with $ b(L)=(1,1). $ Then there exists $ x\in L $ such that $ b(x)=(1,1). $
\end{lem}
\begin{proof}
Suppose, to the contrary, that there is no element  $ x\in L $  such that $ b(x)=(1,1). $ Since $ b(L)=(1,1), $ we have $ x\in L $ such that $ b(x) $ equals $ (2,0) $ or $ (0,2). $ Assume that $ b(x)=(2,0). $ Thus there exist $ y,z\in L$ such that the set $ \lbrace [x,y], [x,z] \rbrace \subseteq L_{\overline{0}}$ is linearly independent. On the other hand,   $ b(L)=(1,1) $ and since there is no element $ x $ such that $ b(x)=(1,1), $ 
 hence there are $ y_1, y_2 \in
 L\setminus \langle x,y,z \rangle$ such that $ [y_1, y_2]\in L_{\overline{1}} $ and  $ [y_1, y_2]\neq 0. $ Without loss generality, let $ y_1\in L_{\overline{0}} $ and $ y_2\in L_{\overline{1}}. $ If $ x, y, z\in L_{\overline{0}}, $ then $ b(x+y_1)=(3, 0).$
 If $ x, y, z\in L_{\overline{1}}, $ then $ b(x+y_2)=(3, 0).$
 This is a contradiction.  If $ b(x)=(0,2), $ then this yields a contradiction in the same way. Therefore the result follows.
\end{proof}
\begin{prop}\label{pr3.6}
Let $ L $ be a finite-dimensional nilpotent Lie superalgebra with $ b(L)=(1,1) $   and let $ A $ be a maximal abelian ideal such that $ b_{A}(L)=(0,0). $ Then one of the following conditions holds. 
\begin{itemize}
\item[(i).] $ \dim [L,L]=(1,1), $
\item[(ii).] $ \dim [L,L]=(1,2), $   $ \dim L/Z(L)=(1,2) $ and there is no element $ w\in L $ such that $ [w,w]\neq 0.$   
\end{itemize}
\end{prop}
\begin{proof}
Let $ x\in L\setminus Z(L). $  Since $ b(L)=(1,1) $ and $ b_{A}(L)=(0,0), $ we have $ A=Z(L) $ and 
 $ L=H\oplus Z(L) $ where $ H $ is a supersubalgebra of $ L $ and 
\begin{align}\label{m}
b(x)&=\dim L-\ker ad_{x}\cr
& = \dim H-\ker ad_{x}|_{H} \leq 2.
\end{align}
Hence there exists  $x\in  H $  with  $ b(x)=(1,1) $  by Lemma \ref{ll}.  We consider the following cases.\\
Case 1. Let $ x\in L_{\overline{0}} $ or $ x\in L_{\overline{1}}. $  
Since $ [L_{\overline{0}}, L_{\overline{0}}]=0 $ by Lemma \ref{lem3.3}, we have  $ x\in L_{\overline{1}}. $ \\
Subcase 1-1. Let $ [x,x]\neq 0. $ Then there exists $  y \in H\setminus \langle x \rangle $ such that the set $ \lbrace [x,y], [x,x] \rbrace$ is linearly independent. Since $ x\in L_{\overline{1}}, $ we have $ y\in  L_{\overline{0}}$ and so $ [y, y]=0. $
Also, the set  $\lbrace [x, L\setminus \langle x,y \rangle], [x,x], [x,y]\rbrace$ is linearly dependent by the inequality \eqref{m}; otherwise $ b(x)\geq 3 $, which is a contradiction. 
 If $ y_1, y_2 \in H \setminus \langle x \rangle $ and $ \lbrace [x,x], [x, y], [y_1, y_2] \rbrace $
is  linearly  independent, then $ b(x+y_1)=3. $ This is a contradiction; hence $ \dim [L, L]=(1,1). $\\
 Subcase 1-2. Let $ [x,x]=0. $ Then there  are $ y, z\in H\setminus \langle x \rangle $ such that $ \lbrace [x, y], [x, z] \rbrace$ is linearly independent.
 We claim that there are no $ y_1, y_2 \in H \setminus\langle  x, y, z \rangle $ such that $ \lbrace [x,y], [x,z], [y_1, y_2] \rbrace $ is linearly
   independent.  Suppose that the set $ \lbrace [x,y], [x,z], [y_1, y_2] \rbrace $ is linearly independent. Then $ b(x+y_1)=3 $, which is a contradiction. If 
   $ b(y)=1 $ and $ b(z)=1, $ then $ \dim [L,L]=(1,1). $ If $ b(y)=2 $ and the set $ \lbrace [x,y], [x,z], [y_1, y] \mid y_1\in H\setminus \langle x, y, z \rangle \rbrace $ is linearly independent, then $ b(x+y_1)=3 $, which is a contradiction. If $ b(z)=2 $ and the set $ \lbrace [x,y], [x,z], [y_1, z] \mid  y_1\in H\setminus \langle x, y, z \rangle \rbrace $ is linearly
    independent, then this yields a contradiction by a similar argument. The only remaining case is when $ [y, z]\neq 0, $ $ [y,z]\in L_{\overline{1}} $ and $  \lbrace [x,y], [x,z], [y, z] \rbrace  $ is linearly independent. Hence $ \dim [L,L]=(1,2). $ 
We show that $[x, y_1]=0,$  $[y, y_1]=0$ and $[z, y_1]=0$ for all $ y_1 \in  L\setminus \langle x, y, z \rangle. $ Suppose, to the contrary, that  
$[x, y_1]\neq 0.$ Then $ y_1- z-y \in \ker ad_x\subseteq \langle Z(L), H\setminus \langle  y, z \rangle \rangle.$ This is a contradiction. The same argument gives $[y, y_1]=0$ and $[z, y_1]=0$ for all $ y_1 \in  L\setminus \langle y_1, y, z \rangle. $
 We now claim that $[y_1, y_2]=0$ for all $ y_1, y_2\in  L\setminus \langle x, y, z \rangle. $ Suppose, to the contrary, that there exists $  y_1, y_2 \in L\setminus \langle x, y, z \rangle  $ and $[y_1, y_2]\neq 0.$ Then $ b(x+y_1)=3, $  $ b(y+y_1)=3 $ or $ b(z+y_1)=3 $, which is a contradiction.
 Hence     $ \dim L/ Z(L)=(1,2). $  Without loss of generality, let $ y\in L_{\overline{0}} $ and $ z\in L_{\overline{1}}. $ Then $ [y,y]=0. $ On the other hand, if $ [z, z]\neq 0, $ then $ \ker ad_{z} |_{H}=0$ and so $ b(z)=\dim L-\dim Z(L)=3. $ This is a contradiction.  
 Hence $ [x,x]=[y,y]=[z,z]=0 $ and so $ [w,w]=0 $ for all $ w\in L. $ Otherwise, there exists $ w\in  L$ such that $ b(w)\geq 3 $ and this is a contradiction.  Therefore $ \dim [L,L]=(1,2), $   $ \dim L/Z(L)=(1,2) $ and there is no element $ w\in L $ such that $ [w,w]\neq 0.$    \\
Case 2.   Assume that $ b(x+y)=(1,1) $ such that $ y \in L_{\overline{0}},$ $ x\in L_{\overline{1}} $ and 
there is no element in either $ H_{\overline{0}} $ or $ H_{\overline{1}} $ with superbreadth $ (1,1). $ Since $ b(L)=(1,1), $ we consider $ [x,z]\in L_{\overline{1}} $ and $ [y,w]\in L_{\overline{0}} $ for $ x,y,z,w\in H $ such that $ \lbrace y, w\rbrace $ is linearly independent and we obtain $ [x+y, x+y]=0 $ and $ b(x+y)=(1,1). $ On the other hand, since there is no element in either $ H_{\overline{0}} $ or $ H_{\overline{1}} $ with superbreadth $ (1,1), $ clearly $ [x, z]=[x,w]=[y,z]=[y,w]=0. $ If there are $ x_1, y_1\in H \setminus \langle x, y, z, w \rangle $ such that  $ [x_1, y_1] \in L_{\overline{0}} $ and  $ \lbrace  [x,y], [x,z], [x_1, y_1] \rbrace $ is linearly independent, then $ b(y+x_1+z)=3 $ or $ b(y+w)=3. $  When  $ [x_1, y_1] \in L_{\overline{1}} $ and  $ \lbrace  [x,y], [x,z], [x_1, y_1] \rbrace $ is linearly independent, then $ b(x+x_1+z)=3 $ or $ b(y+z)=3. $ Therefore $ \dim [L,L]=(1,1). $
\end{proof}


Let $ L $ be a finite-dimensional nilpotent Lie superalgebra and let $ A $ be an ideal of $ L. $  Recall that the centralizer of $ A $ in $ L $ is 
\begin{equation*}
 C_L(A)=\cap _{a\in A} \ker ad_{a} =\lbrace x\in L \mid ad_{x}(a)=0 ~~\forall a\in A \rbrace.
\end{equation*}
For each $ x\in L, $ define  $ M_x=C_L(A)+\ker ad_x $  and $ D_x=\cap_{a\in C_L(A)} M_{a+x}$ and  let 
\begin{equation*}
T_A=\lbrace x\in L \mid b_A(x)=(r,s), r+s=1 \rbrace.
\end{equation*}
\begin{prop}\label{pro3.8}
Let $L $ be a finite-dimensional nilpotent Lie superalgebra, and let $ A $ be a maximal abelian ideal of $ L $ with   $ b_{A}(L)=(r_1, s_1) $ such that $ r_1+s_1=1$ and $ \ker ad_x|_{A}=\ker ad_y|_{A} $ for all $ x, y\in L\setminus C_{L}(A). $  Then 
\begin{itemize}
\item[(i).] $ \dim A/Z(L)=(r', s') $ such that $ r'+s'=1 $ and $ C_{A}(z)=Z(L) $ for all $ z\in T_A. $
\item[(ii).] If $ C_{L}(A) $ is abelian, then $ \dim M_z=\dim L+1-b(z) $ for all $ z\in T_A. $
\item[(iii).] If $ C_{L}(A) $ is non-abelian, then $ \dim M_z>\dim L-b(z) $ for all $ z\in T_A. $
\end{itemize}
\end{prop}
\begin{proof}
(i). We have $ Z(L)=\cap_{x\in L} \ker ad_{x} $  and $ \ker ad_{x}|_{A}=A$ for $ x\in C_{L}(A). $  Then
\begin{align*}
A\cap Z(L)&= A\cap (\cap_{x\in L} \ker ad_{x})\cr
& =\cap_{x\in L} (A \cap \ker ad_{x}) \cr
&=\cap_{x\in L} \ker ad_{x}|_{A}\cr
&=(\cap_{x\in C_{L}(A)} \ker ad_{x}|_{A})\cap (\cap_{x\in L\setminus C_{L}(A)} \ker ad_{x}|_{A})\cr
&=A\cap  \ker ad_{x}|_{A}\qquad x\in T_{A}\cr
&= \ker ad_{x}|_{A} \qquad \qquad x\in T_{A}.
\end{align*}
Hence $ A\cap Z(L)= \ker ad_{x}|_{A}$ for some $ x\in T_{A} $ and so
\begin{align*}
 \dim A/A\cap Z(L)&=\dim A/  \ker ad_{x}|_{A}\quad x\in T_{A}\cr
&=\mathrm{rank}~ ad_{x}|_{A}\qquad\qquad x\in T_{A}\cr
&=b_{A}(x)\cr
&=(r'_1, s'_1)\quad \text{such that }  r'_1+ s'_1=1.
\end{align*}
Since $ A $ is a maximal abelian ideal, we have 
 \begin{equation*}
 \dim A/Z(L)=\dim A/A\cap Z(L)=(r'_1, s'_1) \quad \text{such that }  r'_1+ s'_1=1,
 \end{equation*}
 so
 \begin{align*}
 \dim C_{A}(z)&=\dim \ker ad_{z}|_{A}\quad\qquad\qquad z\in T_{A}\cr
& =\dim A- \mathrm{rank}~ ad_{z}|_{A}\quad z\in T_{A}\cr
 &=\dim A-b_{A}(z)\qquad \qquad \quad z\in T_{A}\cr
 &=\dim A-1\cr
 &=\dim Z(L).
 \end{align*}
 On the other hand, $ Z(L)\subseteq C_{A}(z) $, and hence $ C_{A}(z)=Z(L). $\\
 (ii). Suppose that $ C_{L}(A) $ is abelian. Since $ A $ is a maximal abelian ideal of $ L, $  we have $ A=C_{L}(A). $ Since $ M_{x}=C_{L}(A)+\ker ad_{x} $ for all $ x\in L $ and $ A=C_{L}(A), $ we have 
  \begin{align*}
\dim M_{z}&=\dim  (A+\ker ad_{z})\cr
&=\dim A+\dim \ker ad_{z}-\dim (A \cap \ker ad_{z})\cr
&=\dim A+\dim \ker ad_{z}-\dim \ker ad_{z}|_{A}\cr
&=\dim A+(\dim L-b(z))-\dim C_{A}(z)\cr
&=\dim A+(\dim L-b(z))-\dim Z(L) \cr
&=\dim L+(\dim A-\dim Z(L))-b(z)\cr
&= \dim L+1-b(z)
 \end{align*}
 for all $ z\in T_{A}. $\\
 (iii). Suppose that $ C_{L}(A) $ is non-abelian. Then 
 \begin{align*}
 \dim M_{z}&=\dim  (C_{L}(A)+\ker ad_{z})\cr
&\geq \dim L+1-b(z) \cr
&> \dim L-b(z)
 \end{align*}
 for all $ z\in T_{A}. $
 Hence $ \dim M_{z}> \dim L-b(z) $ for all $ z\in T_{A}. $\\
\end{proof}
\begin{prop}\label{pro3.9}
Let $L $ be a finite-dimensional nilpotent Lie superalgebra  and let $ A $ be a maximal abelian ideal of $ L $ with $ b(L)=(2,0) $  and $ b_A(L)=(1,0) $ such that $ C_{L}(A) $ is non-abelian and $ \dim L/C_{L}(A)=1. $  
\begin{itemize}
\item[(i).]  If $ \dim C_{L}(A)/A=1, $ then $ \dim L/A=2. $
\item[(ii).] If $ \dim C_{L}(A)/A\geq 2, $ then $ \dim [L,L]=(2,0). $
\end{itemize}
\end{prop}
\begin{proof}
(i). Since $ \dim L/C_{L}(A)=1, $ there exists $ u\in L\setminus C_{L}(A) $ such that  $ L=C_{L}(A) \oplus \langle u \rangle. $ On the other hand, $ \dim C_{L}(A)/ A=1, $ so there exists $ v\in C_{L}(A)\setminus A $ such that $ C_{L}(A)=A\oplus \langle v\rangle. $ Therefore $ L/A=\langle \overline{u}, \overline{v} \rangle $ and so $ \dim L/A=2. $ \\
(ii). Let $ \dim C_{L}(A)/ A\geq 2. $ Then there exists a Lie subsuperalgebra  $ H $
of $C_{L}(A)$  such that $ C_{L}(A)=A\oplus H $ and $ \dim H\geq 2. $ Since $ \dim L/C_{L}(A)=1, $
we have
 $ \dim A/Z(L)=1 $ by Proposition \ref{pro3.8} (i) and   there exist $ u\in L\setminus C_{L}(A) $ and  $ v\in A\setminus Z(L). $ Therefore $ L=Z(L)\oplus \langle v\rangle \oplus H\oplus \langle u\rangle. $  Since $ b_{A}(L)=(1,0) $ and $ [A, C_{L}(A)]=0, $ we have $ [v, u]\neq 0 $ and so $ b(v)=(1,0). $ 
If $ b(u)=(1, 0) $ then $ \dim M_{u}> \dim L-1 $ using  Proposition \ref{pro3.8}(iii). Thus $ M_{u}=L $ and $ [u,u]=0. $
 We claim that $ [H,u]=0. $ If there exists $ v_1\in H $ such that $ [u, v_1]\neq 0, $ then $ \ker ad_{u} \subseteq \langle u, Z(L), H\setminus \langle v_1\rangle \rangle .$ 
Since $ b(u)=(1,0), $ the set $ \lbrace  [u, v_1], [u, v]  \rbrace $ is linearly dependent and so $ [u, v_1]= \alpha [u, v] $ for $ \alpha \in \mathbb{F}. $ Hence $ [u, v_1 -\alpha v]=0 $ and so $ v_1 -\alpha v\in \ker ad_{u} \subseteq \langle u,  Z(L), H\setminus\langle v_1  \rangle \rangle .$  This is a contradiction. Thus $ [H,u]=0. $ Since $ b(L)=(2,0), $  $ b(u)=b(v)=(1, 0), $ and $ [u, L\setminus \langle v \rangle]=[v, L\setminus \langle u\rangle]=0, $
there exists  $ v_1\in L $ such that $ b(v_1)=(2,0)$  and $ \lbrace [v_1, v_2], [v_1, v_3] \mid v_1, v_2,  v_3\in H \rbrace $ is linearly independent. We show that $ \dim [L,L]=(2,0). $
On the contrary, the set  $ \lbrace [v_1, v_2], [v_1, v_3], [u,v]\rbrace $ is   linearly independent. Thus $ b(u+v_1)=3, $ and this yields a contradiction. 
Without loss of generality, let $ \lbrace [v_1, v_2],  [u,v]\rbrace $ be   linearly independent. We show that $ \lbrace [v_1, v_2],  [u,v], [a, b] \mid ~~\text{for all}~~ a,b \in L\rbrace $ is   linearly dependent. On the contrary, suppose that $ \lbrace [v_1, v_2],  [u,v], [a, b] \mid ~~\text{for all}~~ a,b \in L\rbrace $ is   linearly independent. Hence $ b(v_1+u+a) $ or $ b(v_2+u) $ equals $ 3 $ and this yields a contradiction.
Therefore $ \dim [L,L]=(2,0). $\\
Let $ b(u)=(2,0)$  and $ [u, u]\neq 0. $ We show that $ [H, u]=0. $
If there exists $ v_1\in H $ such that $ [v_1, u]\neq 0, $ then $ \ker ad_u \subseteq  \langle  A\setminus \lbrace v \rbrace, H\setminus \lbrace v_1 \rbrace  \rangle.$ 
Moreover, since $ b(u)=(2,0) $, the set $ \lbrace  [v_1, u], [v, u], [u,u] \rbrace $ is linearly dependent.  Hence $ [v_1, u]= \beta [v, u]+\gamma [u, u]  $ for $  \beta, \gamma \in \mathbb{F} $ and so $ v_1+\gamma u+\beta v \in \ker ad_{u} \subseteq \langle  A\setminus \langle v \rangle, H\setminus \langle v_1 \rangle  \rangle. $ This is a contradiction. Thus $ [H, u]=0. $ The set $ \lbrace  [v_1,v_2], [v, u], [u,u] \rbrace $ is linearly dependent for all $ v_1, v_2 \in L, $ otherwise $ b(u+v)=(3,0) $, a contradiction. Hence in this case we have $ \dim [L,L]=(2,0). $ \\
Let $ b(u)=(2,0)$ and  $ [u, u]= 0. $ Then there exists $ v_1\in L $ such that $ \lbrace [u,v], [u, v_1] \rbrace$ is linearly independent. Similarly, $ [H\setminus \langle v_1 \rangle, u]=0. $  If  the set $ \lbrace [u,v], [u, v_1], [v_2,v_3]\rbrace$ is linearly independent for all $ v_2, v_3 \in L,$ then we have $ b(u+v_1+v_2)=(3,0) $ or $ b(u+v_3)=(3,0). $ This is a contradiction.  Therefore $ \dim [L,L]=(2,0). $
\end{proof}
\begin{prop}\label{pro3.99}
Let $L $ be a finite-dimensional nilpotent Lie superalgebra  and let $ A $ be a maximal abelian ideal of $ L $ with $ b(L)=(1,1) $  and $ b_A(L)=(r,s) $ such that  $ r+s=1, $ $ C_{L}(A) $ is non-abelian and $ \dim L/C_{L}(A)=1. $  
\begin{itemize}
\item[(i).]  If $ \dim C_{L}(A)/A=1, $ then $ \dim L/A=2. $
\item[(ii).] If $ \dim C_{L}(A)/A\geq 2, $ then $ \dim [L,L]=(1,1). $
\end{itemize}
\end{prop}
\begin{proof}
The proof is analogous to that of Proposition \ref{pro3.9}.
\end{proof}
\begin{prop}\label{prr3.10}
Let $L $ be a finite-dimensional nilpotent Lie superalgebra with $ b(L)=2, $ and let $ A $ be a maximal abelian ideal of $ L $  such that  $ b_A(L)=1$ 
  and $ \dim L/C_{L}(A)\geq 2. $  Then $ \dim [L,L]=2. $
\end{prop}
\begin{proof}
Let $ \dim L/C_{L}(A)\geq 2 $ and  $ [A, x] \cap [A, y] \neq \lbrace 0 \rbrace $ for all $ x, y \in L\setminus C_{L}(A). $ Since $ b_{A}(L)=1, $ we have $ [A, x]=[A,y] $ for $ x, y \in L\setminus C_{L}(A), $  so $ \dim [A, L]=1. $ 
Assume that $ x+[A,L]\in L/ [A,L]. $ Since $ [x,A]\subseteq [x,L], $ $ \dim [x,A]=1$ and $ b(L)=2, $ we have $ \dim [x,L]/ [x,A]\leq 1. $ It follows that $ b(L/[A,L])<2. $ 
 Let $ b(L/[L,A])=0. $ Then $ L/[L,A] $ is abelian by Lemma \ref{l0} and so $ b(L)=1 $, which is a contradiction.
 Assume that $ b(L/[L,A])=1. $ Then $ \dim [L,L]/[L,A]=1 $ by Proposition \ref{pr1.1}. Since $ \dim [A, L]=1, $ we have $ \dim [L,L]=2. $ \\
Assume that   there exist $ x, y \in L\setminus C_{L}(A) $ such that 
$ [A, x] \cap [A, y] = \lbrace 0 \rbrace. $  Since $ x,y\in L\setminus C_{L}(A), $ we have $ [A,x]\neq 0 $ and $ [A, y]\neq 0. $ If there exist 
$ a, b\in A\setminus Z(L), $  $ \lbrace a, b \rbrace$  and  $ \lbrace [a,x], [b,y] \rbrace $ are linearly independent, then $ b_{A}(x+y)=2$ and this is a contradiction with  $ b_{A}(L)=1. $ Hence there is only one element  $ a $ such that $ a\in A\setminus Z(L) $ and  $ \lbrace [a,x], [a,y] \rbrace$ is linearly independent.
 Let $ z\in L\setminus C_{L}(A). $ If $ \lbrace [a,x], [a, y], [a, z] \rbrace $ is linearly independent, then $ b(a)=3 $ and this is a contradiction. Hence  $ \lbrace [a,x], [a, y], [a, z] \rbrace $ is linearly dependent and  $ [a,z]=\alpha  [a,x]+\beta [a, y]$ for $ \alpha, \beta \in \mathbb{F} $ and so $ z-\alpha x-\beta y \in C_{L}(A). $ Therefore $ \dim L/C_{L}(A)=2. $ \\
 Case 1. If $ C_{L}(A) $ is non-abelian, then there exists $ z_1\in C_{L}(A)\setminus A. $ Assume that $ z_2 \in C_{L}(A)\setminus A, $ $ b(z_1)=1 $ and $ \lbrace [z_1, z_2],  [a,x], [a, y] \rbrace$ is a linearly independent set. Then $ b(z_1+a)=3 $ a contradiction. Consider  $ b(z_1)=2, $   $ \lbrace  [z_1,z_2], [z_1, z_3] \rbrace $   and  $ \lbrace [z_1, z_3],  [a,x], [a, y] \rbrace$ for $ z_1, z_2, z_3 \in C_{L}(A)\setminus A $ are linearly independent. Then $b(z_1+a)=3, $ which is a contradiction. In the case $ \lbrace [x, z_1], [a,x], [a, y] \rbrace $ is linearly independent, we have $ b(x+a)=3. $ This is a contradiction and so $\lbrace [x, z_1], [a,x], [a, y] \rbrace $ is linearly dependent  for all $ z_1\in C_{L}(A)\setminus A. $ By a similar argument, the sets $\lbrace [y, z_1], [a,x], [a, y] \rbrace $ and
 $\lbrace [a, z_1], [a,x], [a, y] \rbrace $  are  linearly dependent  for all $ z_1\in C_{L}(A)\setminus A. $ Let $ \lbrace [x, y], [a,x], [a, y] \rbrace  $ be  linearly independent.
We now show that $ \dim [L, L]=2. $ 
 Since $ b_{A}(L)=1, $ $ [A,x]\cap [A,y]=\lbrace 0 \rbrace $ and the set  $ \lbrace [x, y], [a,x], [a, y] \rbrace  $ is linearly independent,  so $ \dim M_{x} $ equals $ \dim C_{L}(A)+1 $ or $ \dim C_{L}(A). $ On the other hand, $ b(x)=\dim L- \ker ad_x $ and 
 $ M_x =C_{L}(A)+\ker ad_x,$ thus 
 \begin{align}\label{e3.10}
2&=b(x)\cr
&=\dim L-\dim \ker ad_{x}\cr
& =\dim L-(\dim C_{L}(A)+1-\dim C_{L}(A)+\dim C_{L}(A)\cap \ker ad_{x})\cr
&=1-\dim C_{L}(A)/C_{L}(A) \cap \ker ad_{x}
\end{align}
or
\begin{align}\label{e3.101}
2&=b(x)\cr
&=\dim L-\dim \ker ad_{x}\cr
& =\dim L-(\dim C_{L}(A)-\dim C_{L}(A)+\dim C_{L}(A)\cap \ker ad_{x})\cr
&=2-\dim C_{L}(A)/ C_{L}(A)\cap \ker ad_{x}.
\end{align}
The equality \eqref{e3.10} leads to a contradiction. The equality \eqref{e3.101} implies that $ C_{L}(A)= C_{L}(A)\cap \ker ad_{x} $ and so 
$ [C_{L}(A),x]=0. $ Hence $ [a,x]=0 $, which is a contradiction.
Similarly, we have $ [C_{L}(A),y]=0$ and this yields a contradiction. Therefore $ \lbrace [x, y], [a,x], [a, y] \rbrace  $ is  linearly dependent and so $ \dim [L,L]=2. $\\
Case 2. If $C_{L}(A)  $ is abelian, then $ C_{L}(A)=A. $ Since  $ L/ C_{L}(A)=2,$ $ x,y\in L\setminus C_{L}(A) $ and  $ [x, A]\cap [y, A]=\lbrace 0\rbrace, $ we have $ [A, L\setminus \langle x, y \rangle]=0. $ It remains to show that $ \lbrace [x, y], [a,x], [a, y] \rbrace  $ is  linearly dependent. Now, similarly in Case 1 one sees $ \lbrace [x, y], [a,x], [a, y] \rbrace  $ is  linearly dependent, so $ \dim [L,L]=2. $
\end{proof}
\begin{prop}\label{pro3.10}
Let $L $ be a finite-dimensional nilpotent Lie superalgebra  and let $ A $ be a maximal abelian ideal of $ L $ with $ b(L)=(2,0) $ and $ b_A(L)=(1,0). $ Then one of the following holds.
\begin{itemize}
\item[(i).] $ \dim A / Z(L)=1 $ and $ \dim L/Z(L)\leq 3, $
\item[(ii).] $ \dim [L,L]=(2,0). $
\end{itemize} 
\end{prop}
\begin{proof}
Case 1. Suppose that $ C_{L}(A)$ is non-abelian. If  $ \dim L/C_{L}(A)=1 $ and $ \dim C_{L}(A)/A=1, $ then   $ L=C_{L}(A)\oplus \langle u \rangle $  for $ u\in L\setminus  C_{L}(A)$ and hence, by Proposition \ref{pro3.8}(i) and the third isomorphism theorem, we have  
\begin{equation}\label{e3.11}
\dim L/Z(L)=\dim L/A+\dim A/Z(L)=\dim L/A+1. 
\end{equation} 
On the other hand, we know that $ \dim L/A=2 $ by Proposition   \ref{pro3.9}(i). Thus  $ \dim L/Z(L)=3. $\\
If  $ \dim L/C_{L}(A)=1 $ and $ \dim C_{L}(A)/A\geq 2, $ then $ \dim [L, L]=(2,0) $ by Proposition \ref{pro3.9}(ii). \\
Let $ \dim L/C_{L}(A)\geq 2. $ Then   $ \dim [L,L]=(2,0) $ by Proposition \ref{prr3.10}. 
 \\
Case 2. Suppose that $ C_{L}(A) $ is abelian. Since $ A $ is a maximal abelian ideal and $ A\subseteq C_{L}(A), $ we have $ C_{L}(A)=A. $ 
Let $ \dim L/A\geq 2. $ Then   $ \dim [L,L]=(2,0) $ by Proposition \ref{prr3.10}. \\
If $ \dim L/A=1, $ then $ L=A\oplus \langle x\rangle. $ 
 Also, $ \dim A/Z(L)=1 $ by Proposition \ref{pro3.8}(i).
Hence  $ \dim L/Z(L)= 2 $ using \eqref{e3.11}. \\
 Therefore, (i) and (ii) follow from Cases 1 and 2.
\end{proof}
\begin{prop}\label{pro.131}
Let $L $ be a finite-dimensional nilpotent Lie superalgebra  and let $ A $ be a maximal abelian ideal of $ L $ with $ b(L)=(0,2) $ and $ b_A(L)=(0,1). $ Then $ \dim [L,L]=(0,2). $
\end{prop}
\begin{proof}
Since  $ C_{L}(A)=A $ by Lemma \ref{lem3.1}, we have $ \dim L/Z(L)=2 $ or $ \dim [L,L]=(0,2) $ by an argument analogous to that in the proof of Case 2 of  Proposition  \ref{pro3.10}. Assume that  $ \dim L/Z(L)=2. $ Since $ b(L)=(0,2), $ we have $ \dim [L,L]=\dim [L_{\overline{0}}, L_{\overline{1}}]. $ Hence $ \dim L/Z(L)=(1,1) $ and so $ L/Z(L)=\langle \overline{v}_{1}, \overline{v}_{2} \rangle.$  Let $ L=\langle Z(L), v_1, v_2 \mid  v_1\in L_{\overline{0}}, v_2\in L_{\overline{1}} \rangle.$ Then $ \dim [L,L]\leq 1 $, which is a contradiction by Lemma \ref{l0} and Proposition \ref{pr1.1}. Therefore $ \dim [L,L]=(0,2). $
\end{proof}
\begin{prop}\label{pr3.14}
Let $L $ be a finite-dimensional nilpotent Lie superalgebra  and let $ A $ be a maximal abelian ideal of $ L $ with $ b(L)=(1,1) $ and $ b_A(L)=(r,s) $ such that $ r+s=1. $ Then one of the following holds.
\begin{itemize}
\item[(ii).] $ \dim A / Z(L)=(r,s) $ such that $ r+s=1 $ and $ \dim L/Z(L)\leq 3. $
\item[(iii).] $ \dim [L,L]=(1,1). $
\end{itemize} 
\end{prop}
\begin{proof}
 The proof is analogous to that of Proposition \ref{pro3.10} and uses Propositions \ref{pro3.99} and \ref{prr3.10}. 
 \end{proof}

\begin{lem}\label{lem3.15}
Let $L $ be a finite-dimensional nilpotent Lie superalgebra and let $ A $ be a maximal abelian ideal of $ L $ with $ b(L)=b_A(L)=(r,s) $ such that $ r+s=2. $   Then  $ \dim [C_{L}(A), L]=(r,s) $ such that $ r+s=2. $
\end{lem}
\begin{proof}
Since $ b_{A}(L)=(r,s)$ such that $ r+s=2, $ there exists $ x\in L\setminus C_{L}(A)$ such that $ b_{A}(x)=(r,s)$ such that $ r+s=2. $ Clearly, $ C_{L}(A)\subseteq D_x \subseteq L. $ We show that $ D_x=L. $ Since $ [A, C_{L}(A)]=0, $ we have $ b_{A}(x)=b_{A}(x+a)=(r,s)=b(L) $ such that $ r+s=2 $ and for all $ a\in C_{L}(A). $ Hence $ \mathrm{Im}~ad_{a+x}|_{A}=\mathrm{Im}~ad_{a+x} $ for all $ a\in C_{L}(A). $ Hence there exists $ b\in C_{L}(A) $ such that  $ [x+a, y]=[x+a, b] $ for all $ y\in L. $ 
Hence $ y-b\in \ker ad_{a+x} $ and so $ y=b+(y-b)\in C_{L}(A)+\ker ad_{x+a} $ for all $ a\in C_{L}(A). $ Therefore $ D_x=L. $ Suppose that $ a\in C_{L}(A) $ and $ y\in L=D_x, $ we may write $ y=a'+y'=a''+y'' $ such that $ a', a''\in C_{L}(A),$  $y'\in \ker ad_{a+x}$ and $y''\in \ker ad_{x}.$ Then $ a'-a''=y'-y''\in C_{L}(A). $ Also, $ [a+x, y']=0=[x, y''] $ for all $a\in C_{L}(A), $ so $ [a, y]=[a, a'+y']=[a,y']=-(-1)^{\mid y'-y''\mid \mid x \mid}[y'-y'', x]\in [C_{L}(A),x] $ for all $ a\in C_{L}(A). $  Hence $ [C_{L}(A), L]\subseteq [C_{L}(A), x] $ and we have $ \dim [C_{L}(A), L]\leq \dim [C_{L}(A), x]. $ On the other hand, we know that $ \dim [C_{L}(A), x]\leq \dim [C_{L}(A),L]. $ Therefore
$ \dim  [C_{L}(A), x]= \dim [C_{L}(A), L]. $
 Since $ b(L)=b_{A}(x)=2, $ we have 
 $ \dim [C_{L}(A), L]=(r, s) $  such that $ r+s=2. $
\end{proof}
\begin{prop}\label{pr3.16}
Let $L $ be a finite-dimensional nilpotent Lie superalgebra  and let $ A $ be a maximal abelian ideal of $ L $ with $ b(L)=b_A(L)=(r,s) $ such that $ r+s=2. $ Then $ \dim [L,L]=(r,s) $ such that $ r+s=2. $ 
\end{prop}
\begin{proof}
Let $L $ be a Lie superalgebra with $ b(L)=b_A(L)=(r,s) $ such that $ r+s=2. $ We have 
$ \dim [C_{L}(A),L]=(r,s) $ such that $ r+s=2 $  by Lemma \ref{lem3.15} and  $ [C_{L}(A),L] \subseteq [L, L],$ so
it is enough to show that $ [L, L]\subseteq [C_{L}(A),L]. $   Assume that $ y\in L\setminus C_{L}(A). $ Then $ b_{A}(y)\leq 2. $ If $ b_{A}(y)=0, $ then $ y\in C_{L}(A) $ and this is a contradiction. 
If $ b_{A}(y)=2, $ then there are $ a, b\in A $ such that $ \lbrace [a, y], [b, y] \rbrace $ is linearly independent.  Since $  b_{A}(y)=b(y)=2, $
the set $ \lbrace [a, y], [b, y], [x,y] \rbrace  $ for $ x\in L $ such that $ [x, y]\neq 0 $ is linearly dependent. Hence $ [x, y]=\alpha [a, y]+ \beta [b, y] \in [C_{L}(A), L] $ for $ \alpha, \beta \in \mathbb{F}. $ 
Thus $ [y, L]\subseteq [C_{L}(A),L] $ for $ b_{A}(y)=2. $
 It remains to show that $ [y, L]\subseteq [C_{L}(A),L] $ for  $ b_{A}(y)=1. $ We claim that $ \dim L/C_{L}(A)\leq 2. $ Since $ b_{A}(L)=2, $  we have $ 2 \leq \dim A/Z(L) $ and so there exists  $ u\in A\setminus Z(L). $ On the other hand, $ b(u)\leq 2 $, so the set $ \lbrace [x, u], [y, u], [z, u]\rbrace $ is  linearly dependent for $ x,y,z\in L\setminus C_{L}(A). $  Hence $ [x, u]= \alpha [y, u]+\beta [z, u] $ and so $ x-\alpha y-\beta z \in C_{L}(A) $ for  $ \alpha, \beta \in \mathbb{F} $ and so $ \dim L/C_{L}(A)\leq 2. $ Clearly $ \dim L/C_{L}(A)\neq 0. $ Thus $1\leq \dim L/C_{L}(A)\leq 2. $ When $ \dim L/C_{L}(A)=1, $ we have $ L=C_{L}(A)\oplus \langle y \rangle.$ 
 But in this case we have $ b_A(L)=1. $
 Hence $ \dim L/C_{L}(A)=2 $ and so $ L=C_{L}(A)\oplus \langle x, y \rangle. $ Let $ b_{A}(x)=2 $ and $ b_{A}(y)=1. $ If $ b(y)=1, $  then $ [y, L]\subseteq [C_{L}(A),L]. $ 
 If $ b(y)=2 $ and $ [x, y]\neq 0, $ then $ \lbrace [x,a], [x, b] , [x, y] \rbrace$ is linearly dependent. Hence $ [x, y]=\alpha [x,a]+ \beta [x, b] \in [C_{L}(A), L].$ Also, if the set $ \lbrace [x,a], [x, b] , [y, y] \rbrace$ is linearly independent, then $ b(y+a)=3 $ or $ b(y+b)=3. $ This yields a contradiction. Hence  $ [y, y] \in [C_{L}(A), L]. $
 Assume that $ b(y)=2 $ and $ [x, y]=0. $ If $ [C_{L}(A), y]\neq [L,y], $ then 
   then $ [y, y]\neq 0 $ and the set $ \lbrace [x,a], [x, b] , [y, y] \rbrace$ is linearly independent.  Thus $ b(y+a)=3 $ or $ b(y+b)=3. $ This yields a contradiction. Hence  the set $ \lbrace [x,a], [x, b] , [y, y] \rbrace$ is linearly dependent and $ [y,y] \in  [C_{L}(A), L].$
   If $ [C_{L}(A), y]= [L,y], $ then clearly  $ [y,y] \in  [C_{L}(A), L].$
Therefore $ [y, L]\subseteq [C_{L}(A),L] $ for all $ y\in L\setminus C_{L}(A) $ and hence $ \dim [L,L]=\dim [C_{L}(A), L]=(r,s) $  such that $ r+s=2 $ by Lemma \ref{lem3.15}. 
\end{proof}

\section{Main results}
In this section, we classify finite-dimensional nilpotent Lie superalgebras $ L $ with $ b(L)=2. $
\begin{thm}\label{M1}
Let $ L $ be a finite-dimensional nilpotent Lie superalgebra. Then $ b(L)=(2, 0) $ if and only if one of the following conditions holds. 
\begin{itemize}
\item[(i).] $ \dim [L, L]=(2, 0), $
\item[(ii).] $ \dim [L,L]=(3,0) $ and $ \dim L/Z(L)=(0,2), $
\item[(iii).] $ \dim [L,L]=(3,0) $ and $ \dim L/Z(L)=(3,0). $
\end{itemize}   
\end{thm}
\begin{proof}
First, suppose that $ b(L)=(2,0). $ Since $ b(L)\neq (0,0), $  $ L $ is non-abelian, so $ Z(L)\subset L. $ Because $ L $ is nilpotent,
$ Z(L)\neq 0 $ and $ L $ has a maximal abelian ideal. Let $ A $ be a maximal abelian ideal.
We have $ b_{A}(L)\leq b(L)=(2,0), $ so the following cases arise.\\
Case 1. If $ b_{A}(L)=(0,0), $ then $ \dim [L, L]=(2, 0) $ or $ \dim [L,L]=(3,0) $ and $ \dim L/Z(L)=(0,2) $ by Proposition \ref{Pr3.4}.\\
Case 2. If $ b_{A}(L)=(1,0), $ then either  $ \dim L/Z(L)\leq 3 $ or $ \dim [L, L]=(2,0) $  by Proposition \ref{pro3.10}.
  Assume that  $ \dim L/Z(L)\leq 3. $ If $ \dim L/Z(L)=1 $ and $ L/Z(L)=\langle \overline{v} \rangle, $ then $ L=\langle Z(L), v \rangle $ and so $ \dim [L,L]\leq 1. $ This is a contradiction by Lemma \ref{l0} and Proposition \ref{pr1.1}.
 Let $ \dim L/Z(L)=2. $ Then $ \dim L/Z(L) $ can be equal to $ (2,0), $ $ (1,1) $ or $ (0,2). $  If $ \dim L/Z(L)=(2,0) $ and 
 $ L/Z(L)=\langle \overline{v_1}, \overline{v_2}  \rangle, $ then $ L=\langle Z(L), v_1, v_2 \vert v_1, v_2 \in  L_{\overline{0}} \rangle$
  and so $ \dim [L,L]=(1,0). $ This yields a contradiction by Proposition \ref{pr1.1}.\\
  If $ \dim L/Z(L)=(1,1) $ and $ L/Z(L)=\langle \overline{v_1}, \overline{v_2} \rangle, $ then 
  \begin{equation*}
  L=\langle Z(L), v_1, v_2 \vert   v_1\in L_{\overline{0}},  v_2 \in L_{\overline{1}} \rangle
  \end{equation*}
   and since $ \dim [L_{\overline{0}}, L_{\overline{1}}]=0, $ we have $ \dim [L,L]=(1,0). $ This is a contradiction by  Proposition \ref{pr1.1}.
 Consider  $ \dim L/Z(L)=(0,2) $ and $ L/Z(L)=\langle \overline{v_1}, \overline{v_2} \rangle. $ Then 
 $ L=\langle Z(L), v_1, v_2 \vert v_1, v_2 \in  L_{\overline{1}} \rangle$
  and $ \dim [L,L]\leq 3. $ Since $ b(L)=(2,0), $ we have $ \dim [L,L]=(2,0) $ or  $   \dim [L,L]=(3,0) $ and  $ \dim L/Z(L)=(0,2). $ \\
Assume that $ \dim L/Z(L)=3.$ Thus $ \dim L/Z(L) $ is one of $ (3,0), $ $ (2,1), $ $ (1,2) $ or $ (0,3). $ We consider the following cases.\\
Let $ \dim L/Z(L)=(3,0) $ and $ L/Z(L)=\langle \overline{v_1}, \overline{v_2}, \overline{v_3} \rangle. $ Then
\begin{equation*}
L=\langle Z(L), v_1, v_2, v_3 \vert  v_1,v_2, v_3\in L_{\overline{0}}  \rangle,
\end{equation*}
 so $ \dim [L,L]\leq 3. $ Since $ b(L)=(2,0), $ we have $ 2\leq \dim [L,L] \leq 3. $ Hence $ \dim [L, L]=(2,0) $ or $ \dim [L,L]=(3,0) $ and 
 $ \dim L/Z(L)=(3,0). $\\
If $ \dim L/Z(L)=(2,1) $ and $ L/Z(L)=\langle \overline{v_1}, \overline{v_2}, \overline{v_3} \rangle, $ then
\begin{equation*}
L=\langle Z(L), v_1, v_2 \vert  v_1, v_2 \in L_{\overline{0}}, v_3 \in L_{\overline{1}}  \rangle.
\end{equation*}
 On the other hand, $ \dim [L_{\overline{0}}, L_{\overline{1}}]=0 $, hence  $ [v_1, v_3]=[v_2, v_3]=0 $ and $ \dim [L, L]=(2,0). $  In the case  $\dim L / Z(L)=(1,2) $ and  $ L/Z(L)=\langle \overline{v_1}, \overline{v_2}, \overline{v_3} \rangle, $  then 
 $ L=\langle Z(L), v_1, v_2, v_3 \vert  v_1\in L_{\overline{0}}, v_2, v_3 \in L_{\overline{1}}  \rangle$
  and since $ \dim [L_{\overline{0}}, L_{\overline{1}}]=0, $ we have $ [v_1, v_2]=[v_1, v_3]=0. $ Hence $ v_1\in Z(L), $ which yields a contradiction. Therefore this case cannot occur. 
Assume that  $ \dim L/Z(L)=(0,3) $ and $ L/Z(L)=\langle \overline{v_1}, \overline{v_2}, \overline{v_3} \rangle. $  Then 
 $ L=\langle Z(L), v_1, v_2, v_3 \vert  v_1, v_2, v_3 \in L_{\overline{1}}  \rangle$ and $ \lbrace [v_1, v_1], [v_2, v_2], [v_3, v_3], [v_1, v_2], [v_1, v_3], [v_2, v_3] \rbrace $ contains all the brackets. 
Since $ b_{A}(L)=(1,0) $ and $ b(L)=(2,0), $ we have $ \dim A/Z(L)=1 $ by Proposition \ref{pro3.10}. 
Let $ v_1\in A\setminus Z(L). $ Then $ [v_1, v_1]=0. $ 
Since  $ b_{A}(L)=(1,0), $  we have $ \dim \langle [v_1, v_3], [v_2, v_3] \rangle=1.$ Hence $ \dim [L,L]\leq 4. $ If $ \dim [L,L]=(4,0), $ then  $ b(v_3)=(3,0) $ or  $ b(v_2)=(3,0), $ and this leads to a contradiction.
Let $ \dim [L, L]=(3,0). $ If $ [v_2, v_3]=0, $ then $ b(v_2+v_3) $ or $ b(v_1+v_3) $ equals $ (3,0). $ This yields a contradiction. 
If $ [v_2, v_3]\neq 0, $ then  by considering the possible bracket configurations there exists $ x\in L/ Z(L)$ such that $ b(x)=(3,0) $, which is a contradiction.  Also, if $ \dim [L,L]\leq 1, $ we obtain a contradiction by Lemma \ref{l0} and Proposition \ref{pr1.1}.
Therefore $ \dim [L,L]=(2,0). $\\
Case 3. Let $ b_{A}(L)=(2,0). $ Then $ \dim [L, L]=(2,0) $ by Proposition \ref{pr3.16}.\\
Conversely, let $ \dim [L,L]=(2,0). $ We have $ b(L)\leq  \dim [L,L],$ hence $ b(L)\leq 2. $ Also, $ b(L)$ is not equal to $ 0 $ and $ 1 $ by Lemma \ref{l0} and Proposition \ref{pr1.1}. Thus $ b(L)=(2,0). $\\
 If $ \dim [L,L]=(3,0) $ and $ \dim L/Z(L)=(0,2), $ then $ b(L)\leq 3. $ By Lemma \ref{l0} and Proposition \ref{pr1.1}, $ b(L) $ is not equal to  $  0 $ and $ 1. $ Now, if $ b(L)=(3,0) $ then $ 3\leq \dim L/Z(L) $ by Theorem \ref{th2.6} (i). Since $ \dim L/Z(L)=(0,2), $ this is a contradiction. Hence $ b(L)=(2,0). $ \\
  Let $ \dim [L,L]=(3,0) $ and $ \dim L/Z(L)=(3,0). $ Then $ b(L)\leq 3 $ and $ L/Z(L)=\langle \overline{v_1}, \overline{v_2}, \overline{v_3} \rangle. $ Hence   $ L=\langle Z(L),   v_1, v_2, v_3 \mid   v_1, v_2, v_3 \in L_{\overline{0}} \rangle$ and so there does not exist $ x\in L $ such that $ [x, x]\neq 0. $ If $ b(L)=(3,0), $ then  $4\leq \dim L/Z(L) $ by Theorem \ref{th2.6} (ii), which is a contradiction. If $ b(L)\leq 1, $ then this is a contradiction by Lemma \ref{l0} and Proposition \ref{pr1.1}. Therefore $ b(L)=(2,0).$
   The result follows.  
\end{proof}
\begin{thm}\label{M2}
Let $ L $ be a finite-dimensional nilpotent Lie superalgebra. Then $ b(L)=(0, 2) $ if and only if $ \dim [L, L]=(0, 2). $
\end{thm}
\begin{proof}
As in the proof of Theorem \ref{M1}, there exists a maximal abelian ideal $ A. $ Then $ b_{A}(L)\leq b(L). $ If $ b_{A}(L)=(0,0), $ then there does not exist such a Lie superalgebra $ L $  by Proposition \ref{pr3.3}. \\
Let $  b_{A}(L)=(0,1).$ Then $ \dim [L,L]=(0,2) $
  by Proposition \ref{pro.131}. \\
Let $ b_{A}(L)=(0,2). $ Then $ \dim [L, L]=(0,2) $ by Proposition \ref{pr3.16}.\\
  The converse follows similarly from the proof of Theorem \ref{M1}.
\end{proof}
\begin{thm}\label{M3}
Let $ L $ be a finite-dimensional nilpotent Lie superalgebra. Then $ b(L)=(1, 1) $ if and only if one of the following conditions holds. 
\begin{itemize}
\item[(i).] $ \dim [L,L]=(1,1), $
\item[(ii).] $ \dim [L, L]=(1,2), $   $ \dim L/Z(L)=(1,2) $ and there is no element $ w\in L $ such that $ [w,w]\neq 0.$   
\end{itemize}   
\end{thm}
\begin{proof}
As in the proof of Theorem \ref{M1}, there exists a maximal abelian ideal   $ A. $  Since $ b_{A}(L)\leq b(L) $ and  $ b(L)=(1,1), $ the following cases arise.\\
Case 1. If $ b_{A}(L)=(0,0), $ then $ \dim[L, L]=(1,1) $  or  $ \dim[L, L]=(1,2), $ $ \dim L/Z(L)=(1,2) $ and there is no element $ w\in L $ such that $ [w,w]\neq 0,$ 
or $ \dim [L,L]=(1,1)$ by Proposition \ref{pr3.6}. \\
Case 2. Assume that $ b_{A}(L)=(r,s) $ such that $ r+s=1. $  Then $ \dim [L,L]=(1,1) $ or $ \dim L/Z(L)\leq 3 $ by Proposition \ref{pr3.14}. If $ \dim L/Z(L)$ equals $(1,0) $ or $ (0,1), $ then $ \dim [L,L]\leq 1 $ and so this is a contradiction by Lemma \ref{l0} and Proposition \ref{pr1.1}. \\
Let $ \dim L/Z(L)=2. $ Then $ \dim L/Z(L) $ equals $ (2,0), $ $ (0,2) $ or $ (1,1). $ Since $ b(L)=(1,1), $ we have that $ \dim [L_{\overline{0}}, L_{\overline{0}}] $ or $ \dim [L_{\overline{1}}, L_{\overline{1}}] $ and  $ \dim [L_{\overline{0}}, L_{\overline{1}}]  $ are non-zero. When $ \dim L/Z(L)=(2,0) $ and $ L/Z(L)=\langle  \overline{v_1}, \overline{v_2} \rangle,  $ we have
$ L=\langle Z(L), v_1, v_2 \mid v_1, v_2 \in L_{0} \rangle.$
 Hence $ \dim [L_{\overline{0}}, L_{\overline{1}}]=0 $ which is a  contradiction. If $ \dim L/Z(L)=(0,2), $ then this yields a contradiction in the same way. Consider $ \dim L/Z(L)=(1,1) $ and $ L/Z(L)=\langle \overline{v_1}, \overline{v_2} \rangle.$ Then 
$ L=\langle Z(L), v_1, v_2 \mid v_1\in L_{\overline{0}}, v_2 \in L_{\overline{1}} \rangle.$
 Thus $ \dim [L,L]=(1,1). $ \\
Let $ \dim L/Z(L)=3 $ and $ L/Z(L)=\langle   \overline{v_1}, \overline{v_2}, \overline{v_3} \rangle. $ Then $ \dim L/Z(L) $ equals $ (3,0), $ $ (0,3), $ $ (2,1) $ or $ (1,2). $ If $ \dim L/Z(L)=(3,0), $  then
$ L=\langle Z(L), v_1, v_2, v_3\mid v_1, v_2, v_3\in L_{\overline{0}} \rangle. $
 Hence $ \dim [L_{\overline{0}}, L_{\overline{1}}]=0  $ and this is a contradiction. 
By a similar argument, if $ \dim L/Z(L)=(0,3), $ then this yields a contradiction.\\
Assume that $ \dim L/Z(L)=(2,1). $  Then 
$ L=\langle Z(L), v_1, v_2, v_3\mid  v_1,v_2\in L_{\overline{0}},  v_3\in L_{\overline{1}}\rangle $ 
and so $ \dim [L,L] \leq 4. $ If $ \dim [L,L]=4, $ then $ b(v_3)=3 $, which is a contradiction. Let $ \dim [L,L]=3. $ 
Since $ b(L)=(1,1), $  $ \dim [L,L] $ equals $ (2,1) $ or $ (1,2). $
On the other hand,  $ \lbrace [v_1, v_2], [v_1, v_3], [v_2, v_3], [v_3,v_3] \rbrace $ contains all brackets. If $ \dim [L,L]=(2,1), $ then 
$ \lbrace [v_1, v_2], [v_3,v_3] \rbrace \subseteq L_{\overline{0}} $ is linearly independent and $ \lbrace [v_1, v_3], [v_2, v_3] \rbrace \subseteq L_{\overline{1}}$ is linearly dependent.
Now, if $ [v_1, v_3]\neq 0$ and $ [v_2, v_3]=0 $ or  $ [v_1, v_3]=[v_2, v_3]\neq 0, $ then $ b(v_1+v_3)=3. $ This is a contradiction.  If $ [v_1, v_3]=0$ and $ [v_2, v_3]\neq 0, $ then $ b(v_2+v_3)=3 $, which is a contradiction. 
In the case  $ \dim [L,L]=(1,2) $  we similarly obtain a contradiction. Thus this case does not occur.
 \\
Let $ \dim L/Z(L)=(1,2) $. Then
$ L=\langle Z(L), v_1, v_2, v_3 \mid   v_1\in L_{\overline{0}}, v_2, v_3 \in L_{\overline{1}} \rangle.$
   Hence $ \dim [L,L]\leq 5. $
 If $ \dim [L,L]=5, $ then $ b(v_2)=4 $, which is a contradiction. 
 Assume that $ \dim [L,L]=4. $ If $ [v_2, v_3]\neq 0, $ we have $ b(v_1)=3, $ $ b(v_2)=3 $ or $ b(v_3)=3 $ and, since $ b(L)=(1,1), $ this is a contradiction. Hence  $ [v_2, v_3]= 0. $ Thus $ b(v_2+v_3)=3 $, which is a contradiction. Hence this case does not occur.\\
  Let $ \dim [L,L]=3. $ Then $ \dim [L,L] $ is either $ (2,1) $ or $ (1,2). $ If $ \dim [L,L]=(2,1), $ then by considering the possible cases, we have $ b(v_1+v_2+v_3)=3, $ $ b(v_1+v_2)=3 $ or $ b(v_1+v_2)=3. $ This is a contradiction. Similarly if $ \dim [L,L]=(1,2), $ this yields a contradiction. 
Therefore $ \dim [L,L]=(1,1). $
\\
Case 3. Let $ b_{A}(L)=(1,1). $ Then $ \dim [L,L]=(1,1) $ by Proposition \ref{pr3.16}.\\
Conversely, let $ \dim [L,L]=(1,1). $ We have $ b(L)\leq  \dim [L,L],$ hence $ b(L)\leq 2. $ Also, $ b(L)$ is not equal to $ 0 $ and $ 1 $ by Lemma \ref{l0} and Proposition \ref{pr1.1}. Thus $ b(L)=(1,1). $\\
 If $ \dim [L,L]=(1,2) $ and $ \dim L/Z(L)=(1,2) $ and there is no element $ w\in L $ such that $ [w,w]\neq 0,$ then $ b(L)\leq 3. $ 
By Lemma \ref{l0} and Proposition \ref{pr1.1}, $ b(L)$  is not equal to $ 0 $ and $ 1. $ Consider $ b(L)=3. $ Since 
 there is no element $ w\in L $ such that $ [w,w]\neq 0$ and  $ b(L)=3, $ we have  $ 4\leq \dim L/Z(L) $ by Theorem \ref{th2.6} (ii). This is a contradiction. 
Hence $ b(L)=2. $  On the other hand,  $ \dim [L, L]=(1,2)$; thus $ b(L)=(1,1), $ as required. 
\end{proof}

\end{document}